\documentclass[10pt]{article}
\usepackage{stmaryrd}              
\usepackage{geometry}                		
\usepackage[parfill]{parskip}    		

\usepackage{amssymb}
\usepackage{amsmath}
\usepackage{palatino}

\usepackage{graphicx}
\usepackage[english]{babel}
\usepackage{amscd}
\usepackage{amsthm}
\usepackage{enumitem}
\usepackage{color}

\newtheorem{theorem}{Theorem}[section]
\newtheorem{lemma}[theorem]{Lemma}
\newtheorem{proposition}[theorem]{Proposition}

\newtheorem{conjecture}[theorem]{Conjecture}

\newtheorem{remark}[theorem]{Remark}
\newtheorem{claim}[theorem]{Claim}

\DeclareMathOperator*{\E}{\mathbb{E}}
\DeclareMathOperator*{\Var}{\mathrm{Var}}

\DeclareMathOperator*{\dto}{\overset{d}{\longrightarrow}}

\newcommand{\GG}{\mathbb{G}}
\newcommand{\R}{\mathbb{R}}
\newcommand{\Spec}{\mathrm{Spec}}
\newcommand{\FamilyX}{\mathfrak{X}}
\newcommand{\C}{\mathcal{C}}
\newcommand{\aut}{\mathrm{aut}}
\newcommand{\Hy}{\mathcal{H}}
\newcommand{\Xn}{X_{n}(H,W)}
\newcommand{\x}{\mathbf{x}}
\newcommand{\sigr}{\sigma_{r,W}}
\newcommand{\rhor}{\hat {\sigma}_{r,W}}
\newcommand{\prob}[1]{\mathbb{P}\left\{ #1 \right\}}
\newcommand{\vwa}[1]{V_{W}^{(#1)}}
\newcommand{\vw}{\vwa{r}}
\newcommand{\ktwok}{K_{r}\ominus_{2}K_{r}}
\newcommand{\dw}{d_{\mathrm{Wass}}}
\newcommand{\mgf}{moment-generating function}
\newcommand{\kk}{\mathbf{k}}
\newcommand{\By}[2]{\overset{\mbox{\tiny{#1}}}{#2}}
\newcommand{\ByRef}[2]{   \By{\eqref{#1}}{#2} }

\newcommand{\eqByRef}[1]{ \ByRef{#1}{=} }

\newcommand{\leByRef}[1]{ \ByRef{#1}{\le} }

\newcommand{\justify}[1]{\fbox{\tiny{#1}}\quad}

\newcommand{\tf}[2]{$#1$} 

\title{A limit theorem for small cliques in inhomogeneous random graphs}

\author{Jan Hladk\'y \thanks{Institute of Mathematics of the Czech Academy of Sciences, \v{Z}itn\'a 25,
		115 67,   Praha 1, Czechia. 
		Research supported by the Alexander von Humboldt Foundation, GA\v{C}R project GJ18-01472Y and RVO: 67985840. E-mail: hladky@math.cas.cz} 
\and
Christos Pelekis\thanks{Institute of Mathematics of the Czech Academy of Sciences, \v{Z}itn\'a 25,
115 67,   Praha 1, Czechia. 
Research supported by GA\v{C}R project GJ18-01472Y and RVO: 67985840. E-mail: pelekis.chr@gmail.com} 
\and 
Matas \v{S}ileikis\thanks{Institute of Computer Science of the Czech Academy of Sciences, Pod Vod\'{a}renskou v\v{e}\v{z}\'{\i} 2, 182 07 Prague, Czechia. Research partially supported by the Czech Science Foundation, grants GJ16-07822Y, GA19-08740S and GJ20-27757Y. With institutional support RVO:67985807. E-mail: matas.sileikis@gmail.com}		
}

\begin{document}
\maketitle
	
\begin{abstract}
	The theory of graphons comes with a natural sampling procedure, which results in an inhomogeneous variant of the Erd\H{o}s--R\'enyi random graph, called $W$-random graphs. We prove, via the method of moments, a limit theorem for the number of $r$-cliques in such random graphs. We show that, whereas in the case of dense Erd\H{o}s--R\'enyi random graphs the fluctuations are normal of order $n^{r-1}$, the fluctuations in the setting of $W$-random graphs may be of order $0, n^{r-1}$, or $n^{r-0.5}$. Furthermore, when the fluctuations are of  order $n^{r-0.5}$ they are normal, while when the fluctuations are of order $n^{r-1}$  they exhibit either normal or a particular type of chi-square behavior whose parameters relate to spectral properties of $W$.
	
	These results can also be deduced from a general setting [Janson and Nowicki, PTRF~1991], based on the projection method. In addition to providing alternative proofs, our approach makes direct links to the theory of graphons.
\end{abstract}


\maketitle

\section{Introduction}\label{s1}

\subsection{Subgraph counts in random graphs}
The purpose of this work is to investigate the distribution of the number of fixed-size cliques in an inhomogeneous variant of the Erd\H{o}s--R\'enyi random graph $\GG(n,p)$. 
The study of the Erd\H{o}s--R\'enyi random graph (see \cite{Janson_Luczak_Rucinski}) is over a half-century old. A central part in the development of the random graph theory concerns methods for understanding the distribution of subgraph counts. Even though ``subgraphs'' may be large-scale structures, like Hamilton cycles, here we are concerned with counting fixed-sized subgraphs. In particular, we want to describe the (bulk of the) distribution of the random variable that counts the number of copies of a fixed subgraph $H$ as $n$ tends to infinity or, in probabilistic language, to obtain a limit theorem for the distribution of subgraph counts of~$H$. 

This problem has many variants (all copies of $H$, induced copies of $H$, joint distribution for several subgraph counts, \ldots), and a variety of tools have been applied to tackle it, including Stein's method in~\cite{Barbour},
ideas from U-statistics in~\cite{Nowicki_Wierman}, and the method of moments in~\cite{Rucinski}. We refer the reader to~\cite{Janson_Luczak_Rucinski} for an entire chapter devoted to the topic and for further references.

Given graphs $H$ and $G$, let $N(H,G)$ denote the number of copies of $H$ in $G$, i.e., the number of subgraphs of $G$ that are isomorphic to $H$, and consider a random variable
\[
X_n :=N(H,\GG(n)) \, ,
\]
where $\GG(n)$ is a model of random graphs on $n$ vertices of interest.
The asymptotic normality of $X_n$ in the Erd\H{o}s--R\'enyi random graph model $\GG(n)=\GG(n,p(n))$, has been fully described by Ruci\'nski \cite{Rucinski} already in 1988. When $p(n)$ is constant it first appeared as a special case of a result by Nowicki and Wierman~\cite{Nowicki_Wierman} who showed that for any fixed graph $H$ with at least one edge and any $p \in (0,1)$,
  \begin{equation*}
    \frac{X_n - \E [X_n]}{\sqrt{\Var [X_n]}}  \; \dto \; Z \sim \mathcal N(0,1).
  \end{equation*}
  (Here and below $Z_n \dto Z$ denotes convergence in distribution.)
  Our result extends this result in the case of cliques $H = K_r$, $r \ge 2$ and fixed $p$ by generalizing $\GG(n,p)$ to the random graph model $\GG(n,W)$ which plays a key role in the theory of dense graph limits introduced in~\cite{BorgsGraphLimits,LovSzGraphLimits}. This model, which we formally define in Section~\ref{ss:inhomo}, can be described as follows: each vertex is assigned a random \emph{type} and then each pair of vertices is independently included as an edge in $\GG(n,W)$ with a probability that depends only on the types of the two vertices. The random graph $\GG(n,p)$ corresponds to a single type and hence the same probability $p$ for every pair of vertices. 

Henceforth, we shall write $\Xn :=N(H,\GG(n,W))$. In Theorem~\ref{thm:main}, we state our main result, a limit theorem for $\Xn$ for each $H = K_r, r \ge 2$. However, since our main result Theorem~\ref{thm:main} requires quite a few additional definitions, for a preview we state its implicit version, Theorem~\ref{thm:implicit}. We start the next subsection by the minimum number of definitions.

\subsection{Inhomogeneous random graphs and the simplified result statement}\label{ss:inhomo}

A \emph{graphon} is a symmetric Lebesgue measurable function $W:[0,1]^{2}\to[0,1]$. Graphons arise as limits of sequences of large finite undirected graphs with respect to the so-called cut metric (see \cite[Part 3]{Lovasz}). Intuitively, graphons may be thought of as graphs on the vertex set $[0,1]$ with infinitesimally small vertices and with a $W(x,y)$-proportion of all possible edges being present in the bipartite graph whose color classes are formed by a small neighbourhood of $x$ and of $y$, respectively. 

Graphons come with  a natural sampling procedure, which results in an inhomogeneous variant of the Erd\H{o}s--R\'enyi random graph. More precisely, given a graphon $W$, the random graph $\GG(n,W)$ is a finite simple graph on $n$ vertices, labelled by the set $[n]:=\{1,\ldots,n\}$, which is generated in two steps: 
in the first step we draw $n$ numbers \emph{types}) $U_1,\ldots, U_n$ independently from the interval $[0,1]$ according to the uniform distribution and we identify their index set with the labels of the vertex set of $\GG(n,W)$;
in the second step, each pair of vertices $i$ and $j$ in $\GG(n,W)$ is connected independently with probability $W(U_{i,}U_{j})$. 
Notice that if $W(x,y)$ is constant, say, $p\in[0,1]$, then $\GG(n,W)$ is the same as the Erd\H{o}s--R\'enyi random graph $\GG(n,p)$. Inhomogeneous random graphs $\GG(n,W)$ provide substantial additional challenges compared to $\GG(n,p)$. For example, while a standard second moment argument shows that the clique number of $\GG(n,p)$ satisfies $\omega(\GG(n,p))\sim\frac{2\log n}{\log (1/p)}$, extending this formula to $\GG(n,W)$ required new techniques,~\cite{DoHlMa}. Further work on inhomogeneous random graphs so far (\cite{BoJaRi,FrMi}) was done in a more general, possibly sparse, model which we mention in Section~\ref{sec:concluding}.

Corollary~10.4 in~\cite{Lovasz} implies that $\Xn$ obeys the law of large numbers, that is, for every $\epsilon>0$,
	\begin{equation}
		\label{eq:LLN}
		\lim_{n \to \infty}\prob{	\Xn = (1 \pm \epsilon)\E [X_n(H,W)] } = 1.
	\end{equation}
This is one of the key results in the theory of limits of dense graph sequences because it shows that each graphon can be approximated by finite graphs with similar subgraph densities. In this article we aim to understand the nature of fluctuations of $\Xn$ around its expectation.

  Fix $r\ge 2$ and write $X_n = X_n(K_r,W)$ for the number of $r$-cliques in $\GG(n,W)$. Since every $r$-set of vertices in $\GG(n,W)$ induces a random graph distributed as $\GG(r,W)$, every $r$-set of vertices induces a clique with probability  
\[
  t(K_r, W) := \prob{\GG(r,W) = K_r}.
\]
If we prescribe type $x$ to one of the vertices in an $r$-set, then this $r$-set induces a clique with probability
\[
  t(x) = t_x(K_r,W) := \prob{\GG(r,W) = K_r \;\big|\; U_1 = x} \;.
  \]
  Clearly, $t(K_r,W) = \int_0^1 t(x) dx$. 

  For a measurable function $f : [0,1]^k \to \R$ we say $f$ is \emph{constant} and denote $f \equiv c$ whenever $f(x) = c$ for almost every $x \in [0,1]^k$.
  
We start with a simplified version of our main result. The additional information provided by the complete version, that is, Theorem~\ref{thm:main}, is the description of constants $\sigma$ and $c_0, c_1, \dots$ in terms of the graphon $W$.
\begin{theorem}\label{thm:implicit}
  For $t(K_r, W), t(x)$, and $X_n$ defined above, the following holds.
\begin{enumerate}[label=(\alph*)]
  \item If $t(K_r,W) = 0$ or $t(K_r,W) = 1$, then almost surely $X_n = 0$ or $X_n = \binom {n}{r}$, respectively. 
  \item  If $t(x)$ is not constant, then there is a constant $\sigma > 0$ such that 
	\begin{equation}
	  \frac{X_n- \E [X_n]}{ n^{r - 1/2}}\;\dto\; \sigma Z \;, 
	\end{equation}
	where $Z$ is a standard normal random variable. 
      \item\label{en:imp_chisquare} If $t(x)$ is constant (but other than $0$ or $1$), then there are real numbers $c_0, c_1, \dots$, such that $\sum_i c_i^2 \in (0, \infty)$ and
	\begin{equation}\label{eq:nonnormal}
	  \frac{X_n - \E [X_n]}{n^{r - 1}} \;\dto\; c_0 Z_0 + \sum_{i \ge 1} c_i (Z_i^{2}-1) \;,
\end{equation}
where $Z_0, Z_1, \dots$ are independent standard normal random variables.
The series on the right-hand side of~\eqref{eq:nonnormal} converges a.s.\ and in $L^{2}$.
 \end{enumerate}
\end{theorem}
 \begin{proof}
   Theorem~\ref{thm:main} implies all the claims of Theorem~\ref{thm:implicit} except for the fact that $\sum_i c_i^2 > 0$, which follows from Remark~\ref{rem:nondegen}.
 \end{proof}

We note that non-normal limit theorem occurs even in $\GG(n,p)$ with $p$ fixed, but for \emph{induced} subgraph counts, see \cite[Theorem 6.52]{Janson_Luczak_Rucinski}.

\paragraph{A toy instance of Theorem~\ref{thm:implicit}\ref{en:imp_chisquare}.}
The proof of Theorem~\ref{thm:implicit}\ref{en:imp_chisquare} proceeds by the method of moments and thus in itself does not provide much intuition for the asserted non-normal limit fluctuations. This behavior is however suggested by the following simple example. Let $r\ge 2$ be arbitrary. Consider $W$ which is a `disjoint union of two equally-sized cliques', that is $W(x,y) = 1$ if $(x,y)\in [0,\frac{1}{2})^2\cup [\frac{1}{2},1]^2$ and $W(x,y) = 0$ otherwise. It is easy to verify that $t(x) = 2^{ - r + 1}$ for each $x\in [0,1]$, in particular, $t(x)$ is constant. Now, the graph $\GG(n,W)$ is determined already after the first step of the construction, that is, after sampling $U_1,\ldots, U_n$. Indeed, $\GG(n,W)$ will consist of two disconnected cliques, one consisting of the vertices $V_1: = \{i: U_i\in [0,\frac12)\}$ and the other one consisting of the vertices $V_2: = \{i: U_i\in [\frac12,1]\}$. By the central limit theorem, we have $|V_1|\approx\frac{n}{2} + \mathsf{N}\cdot \sqrt{n}$, where $\mathsf{N}$ is a normal random variable with mean~0 and variance~$\tfrac{1}{4}$. Deterministically, $|V_2| = n - |V_1|\approx\frac{n}{2} - \mathsf{N}\cdot \sqrt{n}$. Hence, the number of $K_r$'s in  $\GG(n,W)$ is equal to
\[
\binom{|V_1|}{r} + \binom{|V_2|}{r}\approx \binom{\frac{n}{2} + \mathsf{N}\cdot \sqrt{n}}{r} + \binom{\frac{n}{2} - \mathsf{N}\cdot \sqrt{n}}{r}\approx \frac{1}{r!}\cdot \left(\left(\frac{n}{2} + \mathsf{N}\cdot \sqrt{n}\right)^r + \left(\frac{n}{2} - \mathsf{N}\cdot \sqrt{n}\right)^r\right)
\;.
\]
The first terms in the binomial expansions of these two summands add up to $\frac{2n^r}{2^r}$ which is non-random (and hence does not contribute to fluctuations). The second terms in the binomial expansions cancel out. Hence the main part of the fluctuations comes from the sum of the third terms, 
\[\binom{r}{2}(\tfrac{n}{2})^{r - 2}(\mathsf{N}\cdot \sqrt{n})^2 + \binom{r}{2}(\tfrac{n}{2})^{r - 2}( - \mathsf{N}\cdot \sqrt{n})^2 = 2^{3-r}\binom{r}{2}\cdot \mathsf{N}^2\cdot n^{r - 1}\;.\]
In particular, we see that these fluctuations are of order $n^{r-1}$ and involve the square of a normal random variable.

\paragraph{Connection to generalized U-statistics.}
During the post-submission revision the authors became aware of a result by Janson and Nowicki~\cite{Janson_Nowicki} on orthogonal decomposition of generalized U-statistics of independent random variables indexed by both vertices and edges of a graph (see also~\cite{KR}). In particular, Theorem 2 in \cite{Janson_Nowicki} implies Theorem~\ref{thm:implicit}. To see this, one has to realize that the clique count can be expressed as a U-statistic by generating independent uniform random variables $Y_{uv} \in [0,1]$ (independent of $U_1, ..., U_n$) and defining $\GG(n,W)$ as the graph with edge set 
\begin{equation*}
    \{ uv :  Y_{uv} \le W(U_u, U_v)\},
\end{equation*}
where $U_v, v \in [n]$ are the random types as defined above. 

The contribution of this paper is thus (i) a direct proof of limit distribution using the method of moments; (ii) explicit description of the limit distribution in terms of the graphon $W$ (which admittedly also can be read out relatively easily from the result of Janson and Nowicki); (iii) establishing a link between asymptotic normality of clique counts and quasirandomness (see Conjecture~\ref{conj} and the following comment). 

\bigskip 
 
Before stating the explicit version of Theorem~\ref{thm:implicit} we need to introduce some definitions about spectra of graphons and more advanced concepts related to subgraph densities. On the way, we also recall facts that will be useful for the proof of the main theorem.

\subsection{Spectrum of a graphon and cycle densities}\label{sss:spectrum}
Much of the spectral theory of graphs carries over to the dense graph limit setting, where graphons play the role that adjacency matrices of graphs play in algebraic graph theory.
Spectral properties of graphons will be crucial in stating and proving the explicit version of Theorem~\ref{thm:implicit}\ref{en:imp_chisquare}. 

In this section, we follow~\cite[Section 7.5]{Lovasz}, where details and proofs can be found. We work with the real Hilbert space $L^2[0,1]$. Suppose that $W:[0,1]^2\rightarrow [0,1]$ is a graphon. Then we can associate with $W$ its kernel operator $T_W:L^2[0,1]\to L^2[0,1]$ by setting 
\[ (T_W f)(x) = \int_{0}^{1} W(x,y) f(y)\; dy  \]
for each $f\in L^2[0,1]$.
The operator $T_W$ is a Hilbert--Schmidt operator and that has a discrete spectrum. That is, there exists a countable multiset, denoted $\Spec(W)$, of non-zero real \emph{eigenvalues associated with $W$}. Moreover, we have that
\begin{equation}\label{eq:eigenvalues}
\sum_{\lambda\in\Spec(W)} \lambda^2 = \int_{[0,1]^2} W(x,y)^2\; dx\; dy \le 1. 
\end{equation}

The \emph{degree function} of a graphon $W$ is the function $\deg_W : [0,1] \to [0,1]$ defined as $\deg_W(x) = \int_y W(x,y) dy$. In order to gain some intuition about the degree function, the reader should note that if a vertex in $\GG(n,W)$ is conditioned to have type $x$, it has expected degree $(n-1)\cdot \deg_W(x)$. 
We say that $W$ is \emph{regular} if $\deg_W(x) \equiv d$ for some constant $d$. (Note that in such case $d = \int_{[0,1]^2} W(x,y)\;dx\;dy$.) 

Observe that if $W$ is regular, then $f \equiv 1$ is an eigenfunction of $T_W$ with eigenvalue $d$. In this case, let $\Spec^-(W)$ be $\Spec(W)$ with the multiplicity of $d$ decreased by~1. (It can also be shown that all eigenvalues are at most $d$ in absolute value, but this is not necessary for our proof.)

One of the most useful properties of eigenvalues of a graphon is that they give a simple expression for cycle densities. Recall that $C_\ell$ is a cycle on $\ell$ vertices (with $C_2$ being a multigraph consisting of a double edge). We have (see \cite[(7.22), (7.23)]{Lovasz}) that
\begin{equation}\label{eq:cycle_eigenvalues}
  t(C_{\ell},W) = \sum_{\lambda\in \Spec(W)} \lambda^\ell, \qquad \text{for }\ell \ge 2\,.
\end{equation}
Note that for $\ell = 2$, the formula \eqref{eq:cycle_eigenvalues} is exactly \eqref{eq:eigenvalues}.

The formula~\eqref{eq:cycle_eigenvalues} plays a key part in the proof Theorem~\ref{thm:main}(c), and is the only fact that we use from the spectral theory of graphons.

\subsection{Conditional densities, \tf{K_r}{K\_r}-regular graphons and  \tf{\vw}{V\_{W}(r)}}\label{sss:VWR}

If $H$ is a fixed multigraph and $W$ is a graphon, the \emph{density} (usually called \emph{homomorphism density}) of $H$ in $W$ is defined as 
\begin{equation}\label{eq:1}
t(H,W):= \mathbb{E}\prod_{\{i,j\}\in E(H)} W(U_i,U_j) \, . 
\end{equation}
(Notice that if the edge $\{i,j\}$ has multiplicity $m$ in $H$, then the corresponding contribution to the density equals $W(U_i,U_j)^m$.)
When $H$ is a simple graph on $k$ vertices, then the constant $t(H,W)\in[0,1]$ is the probability that a particular copy of $H$ is present in $\GG(n,W)$, which implies
\begin{equation*}
  \E [\Xn] = \frac{(n)_k}{\aut(H)}t(H,W),
\end{equation*}
where $\aut(H)$ is the number of automorphisms of $H$, and $(n)_k = n\cdot (n-1)\cdot\ldots\cdot (n-k+1)$.  
In particular we have that $\mathbb{E}X_{n}(K_r,W) = \binom{n}{r} t(K_r,W)$, a fact that will be used several times throughout the paper.

For a natural number $k$, we write $[k]:=\{1,2,\ldots,k\}$. Given an integer $\ell\le k$, let ${\binom{[k]}{\ell}}$ denote all $\ell$-element subsets of $[k]$. Let $J\in {\binom{[k]}{\ell}}$ and suppose that $H$ is a graph on the vertex set $[k]$ for which the vertices from the set $J$ are considered as \emph{marked}. 
Given a vector $\x = (x_j)_{j \in J} \in [0,1]^{\ell}$, we define 
\begin{equation}\label{eq:2}
	t_{\x}(H,W) = \E\left[ \prod_{\{i,j\}\in E(H)} W(U_i,U_j)\;\bigg|\; U_j=x_j: j\in J\right]\;. 
\end{equation}
Again, if $H$ is simple with $|V(H)| = r$, then $t_{\x}(H,W)$ is the conditional probability that $\GG(r,W) = H$ whenever vertex $j$ is prescribed a type $x_j$, for $j\in J$.
Note that, when $H=K_r$ is the $r$-clique, the function $\x \mapsto t_{\x}(K_r,W)$ depends only on the cardinality of $J$ (and not on $J$ itself). In this case, we write $K_{r}^{\bullet}$ and $K_{r}^{\bullet\bullet}$ for $K_{r}$ with one, respectively two, marked vertices and denote the corresponding conditional densities by $t_x(K_{r}^{\bullet},W)$ and $t_{x,y}(K_{r}^{\bullet\bullet},W)$.

A graphon $W$ is called \emph{$K_r$-free} if $t(K_r,W)=0$ and called \emph{complete} if $W$ equals $1$ almost everywhere. 

We say that $W$ is $K_{r}$-\emph{regular} if for almost every $x\in[0,1]$ we have 
\begin{equation*}
	t_{x}\left(K_{r}^{\bullet},W\right)=t(K_{r},W)\;.
\end{equation*}
In the case $r = 2$, we have $t_{x}\left( K_{r}^{\bullet}, W \right) = \deg_W(x)$, hence $K_2$-regularity coincides with the usual concept of regularity. For $r \ge 3$ condition of $K_r$-regularity means that in the random graph $\GG(n,W)$ a vertex is expected to belong to the same number of copies of $K_r$, regardless of its type. Another important aspect of $K_r$-regularity of $W$ is that any two particular copies of $K_r$ that share exactly one vertex are uncorrelated, that is, in $\GG(n,W)$ existence of one of the these two copies does not influence the probability of the existence of the other. As we will see in the proof, if $W$ is \emph{not} $K_r$-regular, then a pair of $K_r$-copies sharing one vertex are \emph{positively} correlated (see \eqref{eq:d1tr}), intuitively causing larger variance of the $K_r$ count.

As we will see in Theorem~\ref{thm:main}, the limit distribution of edges (i.e. case $r = 2$) can be described in terms of the function $\deg_W$ and the spectrum of $W$. For $r \ge 3$, however, we need to consider an auxiliary graphon $\vw$ defined below which encodes the information about the local clique densities in $W$.

Suppose that $W$ is a graphon and $r\ge 2$. Then we define a graphon $\vw:[0,1]^2\to [0,1]$ by setting 
\begin{equation}\label{eq:Volkswagen}
	\vw(x,y) := t_{x,y}(K_{r}^{\bullet\bullet},W).
\end{equation}
So, $\vw(x,y)$ is intuitively the density of $K_r$'s containing $x$ and $y$. Note that~\mbox{$\vwa{2}=W$}.

Suppose that we have two numbers $r\in\mathbb{N}$ and $j\in\{0,\ldots,r\}$. We write $K_{r}\ominus_{j}K_{r}$ for the (simple) graph on $2r-j$ vertices consisting of two copies of $K_{r}$ sharing $j$ vertices. For $j = 2$, we also denote by $K_r \oplus_2 K_r$ the multigraph obtained from $K_r\ominus_2 K_r$ by doubling the shared edge. In particular $K_2 \ominus_2 K_2 = K_2$ and $K_2\oplus_2 K_2 = C_2$.

Denote by $K_r'$ and $K_r''$ two copies of $K_r$ that share exactly two vertices, which have labels $1$ and $2$. We have
\[
  \prod_{ij \in K_r \oplus K_r} W(U_i,U_j) = W(U_1,U_2)\prod_{ij \in K_r \ominus_2 K_r} W(U_i,U_j) = \prod_{ij \in K_r'} W(U_i,U_j) \prod_{ij \in K_r''} W(U_i,U_j)
\]
Taking conditional expectation with respect to the event $U_1 = x, U_2 = y$ and noting that products $\prod_{ij \in K_r'} W(U_i,U_j)$ and $\prod_{ij \in K_r''} W(U_i,U_j)$ are conditionally independent, from definition \eqref{eq:2} we obtain 
\begin{equation}
\label{eq:tKWplus}
t_{x,y}(K_r \oplus_2 K_r, W) = W(x,y) t_{x,y}(K_{r}\ominus_{2}K_{r}, W) = \left( t_{x,y}(K_r^{\bullet\bullet}, W) \right)^2 = \left(\vw(x,y)\right)^2.
\end{equation}
	
Let 
\begin{equation}\label{eq:kissmysoul}
  \sigr^2 :=\frac{1}{2((r-2)!)^{2}}\left(t(K_{r}\ominus_{2}K_{r},W)- t(K_{r} \oplus_2 K_{r}, W)\right)\;.
\end{equation}
We have $\sigr^2 \ge 0$, since, assuming that the two shared vertices have labels $1$ and $2$,
\begin{equation}\label{eq:withthesehairs}
  \begin{split}
  t(K_{r}\ominus_{2}K_{r},W) - t(K_r \oplus_2 K_r, W) &= \int (t_{x,y}(K_{r} \ominus_2 K_{r}) - t_{x,y}(K_r \oplus_2 K_r)) \;dx \;dy\\
  &\ByRef{eq:tKWplus}{=} \int [1 - W(x,y)]t_{x,y}(K_r \ominus_2 K_r) \;dx\; dy\\
    &\ge \int 0 \;dx\;dy = 0.
  \end{split}
\end{equation}

  Observe that 
  \begin{equation*}
    \deg_{\vw}(x) = \int_0^1 \vw(x,y) \;dy = \int_0^1 t_{x,y}(K_r^{\bullet\bullet},W) \;dy = t_x(K_r^{\bullet},W).
  \end{equation*}
So $W$ is $K_r$-regular if and only if $\vw$ is regular, with $\deg_{\vw} \equiv t_r := t(K_r,W)$.
Hence, by the remark we made in Section~\ref{sss:spectrum}, one of the eigenvalues associated with $\vw$ is $t_r$. In this case,  
$\Spec^-(\vw)$ is $\Spec(\vw)$ with the multiplicity of $t_r$ decreased by~1.

\subsection{Statement of the main result}\label{ssec:statement}
We are now ready to state our main result. 

\begin{theorem}\label{thm:main}
	Let $W$ be a graphon. Fix $r\ge 2$ and set $t_{r}=t(K_{r},W)$. Let $X_{n,r}=X_{n}(K_{r},W)$ be the number of
	$r$-cliques in $\GG(n,W)$. Then we have the following.
\begin{enumerate}[label=(\alph*)]
  \item\label{en:nothing} If $W$ is $K_{r}$-free or complete then almost surely $X_{n,r} = 0$ or $X_{n,r} = \binom {n}{r}$, respectively. 
\item\label{en:gaussian} If $W$ is not $K_{r}$-regular, then 
	\begin{equation}\label{eq:normality}
		\frac{X_{n,r}-{\binom n r}t_r}{ n^{r-\frac{1}{2}}}\;\dto\; \rhor\cdot Z \;, 
	\end{equation}
	where $Z$ is a standard normal random variable and $\rhor = \frac{1}{(r-1)!}\left(t(K_{r}\ominus_{1}K_{r},W) -t_r^2 \right)^{1/2} > 0$. 

\item\label{en:chisquare}If $W$ is a $K_{r}$-regular graphon which is neither $K_{r}$-free nor complete, then
\begin{equation}\label{eq:chsquareformula}
\frac{X_{n,r}-{\binom n r}t_r}{n^{r-1}} \;\dto\;\sigr\cdot Z+\frac{1}{2(r-2)!}\sum_{\lambda\in\Spec^-(\vw)}\lambda\cdot (Z_\lambda^{2}-1)\;,
\end{equation}
where $Z$ and $(Z_\lambda)_{\lambda\in\Spec^-(\vw)}$ are independent standard normal and $\sigr$ is defined in \eqref{eq:kissmysoul}.
(The series on the right-hand side of~\eqref{eq:chsquareformula} converges a.s.\ and in $L^{2}$ by Lemma~\ref{lem:BhaDiaMuk}.)
 \end{enumerate}
\end{theorem}
Let us comment briefly on Theorem~\ref{thm:main}. Part~\ref{en:nothing} is immediate. Part~\ref{en:gaussian} tells us that in this setting we have a behaviour as in the central limit theorem; indeed we could have stated Part~\ref{en:gaussian} as $(X_{n,r}-\E[X_{n,r}])/\sqrt{\Var[X_{n,r}]}\dto Z$. 
Finally, we will see in Remark~\ref{rem:nondegen} below that the limit in \eqref{eq:chsquareformula} is non-degenerate, implying the Theorem~\ref{thm:main} gives a complete description of limit distributions of small cliques in $\GG(n,W)$.

Let us mention that Part~\ref{en:gaussian} has been recently reported in~\cite{FeMeNi:modGaussian}, using a framework of the so-called mod-Gaussian convergence, developed in that paper. This concept actually gives much more: firstly, the authors establish normal behaviour under conditions analogous to those in Part~\ref{en:gaussian} also for other graphs than $K_r$. Secondly, they also prove a moderate deviation principle and a local limit theorem in this setting. So the reason we provide a proof of Part~\ref{en:gaussian} is that ours  --- based on Stein's method --- is much simpler (because we are proving a weaker statement). But the main emphasis of the paper is on Part~\ref{en:chisquare}, which is new and deals with a regime exhibiting a more exotic behaviour.

\subsection{When the distribution in Theorem~\ref{thm:main}\ref{en:chisquare} is normal or normal-free}
Recall that a chi-square distribution with $k$ degrees of freedom  is the distribution of a sum of the squares of $k$ independent standard normal random variables. Therefore the series in~\eqref{eq:chsquareformula} is a weighted infinite-dimensional variant of a chi-square distribution. By~\eqref{eq:eigenvalues} and \eqref{eq:varS} below, this random variable has finite variance. Interestingly, very similar distributions appear in~\cite{BhaDiaMuk} and ~\cite{BhaMuk}, also in connection with graph limits. That said, the particular setting of our paper seems to be substantially different from~\cite{BhaDiaMuk,BhaMuk}.

\begin{proposition}
  \label{prop:purenormal}
  In Theorem~\ref{thm:main}\ref{en:chisquare} the limit distribution is normal if and only if $\vw \equiv t_r$.
\end{proposition}
\begin{proof}
  In view of~\eqref{eq:chsquareformula}, a normally distributed limit implies that $\Spec^-(\vw)=\emptyset$ and therefore $\Spec(\vw) = \left\{ t_r \right\}$. Recall that $\vw$ is regular with degree function being constant $t_r$. We claim that then $\vw \equiv t_r$. This claim can be viewed as a graphon version of a consequence of Chung--Graham--Wilson Theorem on quasirandom graph sequences\footnote{more precisely, the part ``each spot has the same density $\iff$ non-principal eigenvalues are small'', denoted as $P_4$ and $P_3$ in~\cite[page 158]{AlonSpencer}}; here we give a short self-contained proof. Indeed, regularity of $\vw$ implies $\deg_{\vw} \equiv t_r$, so we have that
  \begin{align*}
    t_r^2&=\left(\int_{[0,1]} \deg_{\vw}(y)\; dy\right)^2=
    \left(\int_{[0,1]^2} \vw(x,y)\; dx\; dy\right)^2\\
    \justify{Jensen's inequality}&\le \int_{[0,1]^2} \left(\vw(x,y)\right)^2\; dx\; dy \eqByRef{eq:eigenvalues} \sum_{\lambda\in \Spec(\vw)}\lambda^2=t_r^2 \, .
  \end{align*}
Since the quadratic function is strictly convex, equality is Jensen's inequality is attained if and only if $\vw$ is constant, which implies $\vw \equiv t_r$. 
\end{proof}
So, the question now is which graphons $W$ lead to a constant graphon $\vw$. 
Since $\vwa{2} = W$, for $r=2$ the answer is clearly given by  constant graphons $W$.
For $r\ge 3$, we put forward the following conjecture, which was first hinted in concluding remarks of~\cite{Reiher_Schacht}.

\begin{conjecture}\label{conj}
  Suppose that $r\ge 3$ and $\vw$ is a constant-$d$ graphon for some $d\in[0,1]$, that is, $t_{\cdot, \cdot}(K_{r}^{\bullet\bullet},W) \equiv d$.
  Then $W$ is $K_r$-free (when $d=0$), or $W \equiv d^{1/{\binom r 2}}$.
\end{conjecture}
In~\cite{Reiher_Schacht}, the case $r=3$ of the aforementioned conjecture was shown to be true. Therefore, we know that if $W$ is a graphon which is $K_3$-regular and not $K_3$-free, then the only way we can get normal limit distribution in Theorem~\ref{thm:main}\ref{en:chisquare} is when $W$ is a constant graphon. Conjecture~\ref{conj} can also be rephrased as follows: among random graphs $\GG(n,W)$, where $W$ is a $K_r$-regular graphon, only $\GG(n,p)$, $p \in (0,1)$, has asymptotically normal count of $K_r$.

Let us now comment on a complementary question: when is the normal
term absent in~\eqref{eq:chsquareformula}? 
\begin{proposition}
  \label{prop:nonormal}
  We have $\sigr = 0$ if and only if $W(x,y)=1$ for almost every $(x,y)\in[0,1]^2$ for which $t_{x,y}(K_{r}^{\bullet\bullet},W)>0$.
\end{proposition}
\begin{proof}
We first observe that $t_{x,y}(K_{r}^{\bullet\bullet},W)>0$ if and only if $t_{x,y}(K_{r}\ominus_{2}K_{r},W)>0$. For $W(x,y) > 0$ this is immediate from the second equality in \eqref{eq:tKWplus}, while for $W(x,y) = 0$ from definition of $t_{x,y}(H,W)$ is is clear that both $t_{x,y}(K_{r}\ominus_{2}K_{r},W) = 0$ and 
  $t_{x,y}(K_{r}^{\bullet\bullet},W) = 0.$  
  From~\eqref{eq:kissmysoul} and \eqref{eq:withthesehairs}, we have $\sigr = 0$ if and only only if $W(x,y)=1$ for almost every $(x,y)\in[0,1]^2$ for which $t_{x,y}(K_{r}\ominus_{2}K_{r},W)>0$. Together with the first observation this proves the proposition.
\end{proof}

For $r=2$ the condition of Proposition~\ref{prop:nonormal} is equivalent to a condition that  $W \in \left\{ 0,1 \right\}$ almost everywhere. 
For $r \ge 3$ we have more freedom for constructions. For example, take $r=3$, partition $[0,1]$ into 6 sets of measure $\frac{1}{6}$ each and put one copy of the complete 3-partite graphon on the first~3 sets and another copy on the last~3 sets. Make arbitrarily wild connections between the 1st and the 4th set, and set the rest of the connections between the first 3 and the last 3 sets to~0 (see Figure~\ref{fig:tri_ex}). Such a graphon $W$ is $K_3$-regular but we have $\sigr = 0$.
\begin{figure}
\begin{center}
		\includegraphics[scale=1]{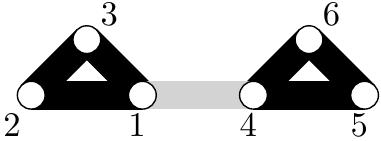}
\end{center}			
	\caption{Example of $K_3$-regular graphon (depicted here as a graph) with no linear term in the non-normal limit.}
	\label{fig:tri_ex}
\end{figure}

  \begin{remark}
    \label{rem:nondegen}
    The limit in~\eqref{eq:chsquareformula} is never degenerate. Indeed, if the non-normal part vanishes, then Proposition~\ref{prop:purenormal} implies $\vw \equiv t_r \in (0,1)$. Further, from \eqref{eq:tKWplus} we obtain that almost everywhere
    \[
    W(x,y)t_{x,y}(\ktwok,W) = \left(\vw(x,y)\right)^2 = t_r^2 > 0.
  \]
    In particular this implies $t_{x,y}(\ktwok,W) > 0$ almost everywhere. Moreover, we have $W < 1$ on a set of positive measure (otherwise $W \equiv 1$ and we would be in the case \ref{en:nothing}), which implies that the inequality in \eqref{eq:withthesehairs} is strict. Since the left-hand side of \eqref{eq:withthesehairs} is a multiple of $\sigr^2$, this implies that $\sigr^2 > 0$.
  \end{remark}
\subsection{Acknowledgment}
We thank the anonymous referee for a very detailed report which helped us to improve the exposition.

\section{Preliminaries}
In this section we state definitions and facts needed in the proof of Theorem~\ref{thm:main}. It can be skipped at the first reading.

Asymptotic notation like $a_n = O(b_n)$ and $a_n \sim b_n$ (equivalently, $a_n = (1 + o(1))b_n$) is stated with respect to $n \to \infty$.
\subsection{Hypergraphs, associated graphs, and further spectral properties}
  \label{ss_hypergraphs}
Given $r \ge 2$, an \emph{$r$-uniform hypergraph} $\Hy$ on vertex set $V$ is a family of $r$-element subsets (called \emph{hyperedges}) of $V$. In this paper we assume that $\Hy$ is a multiset (even though a term \emph{multihypergraph} would be a more standard term). We omit the words ``$r$-uniform'', when this is clear from the context. By $|\Hy|$ we denote the number of hyperedges, counting multiplicities.
The \emph{degree} of a vertex $v\in V$, denoted by $\deg_{\Hy}(v)$, is the number of hyperedges (counting multiplicities) of $\Hy$ containing $v$. We say that $\Hy$ is \emph{spanning} if $V = \cup_{e \in \Hy} e$.
Given a hypergraph $\Hy$, the \emph{graph associated with $\Hy$}, sometimes also called the clique graph of $\Hy$, is a graph on the same vertex set, where each hyperedge $S$ of $\Hy$ is replaced by a clique on $S$, with multiple edges being replaced by single ones. 

We use a particular hypergraph version of cycles, known as loose cycles. For $\ell \ge 2$, let $\C_{\ell}^{(r)}$ be a hypergraph with $\ell$ edges, each of size $r$, created from a graph $C_\ell$ by inserting into each edge additional $r-2$ vertices, all $\ell(r-2)$ new vertices being distinct (hence $\C_{2}^{(r)}$ is a pair of $r$-sets sharing exactly~2 vertices). Finally, let $G_{\ell,r}$ be the graph associated with $\C_\ell^{(r)}$. Note that $G_{2,r} = K_r \ominus_2 K_r$. See Figure~\ref{fig:cycles} for examples.

\begin{figure}
\begin{center}
		\includegraphics[scale=0.8]{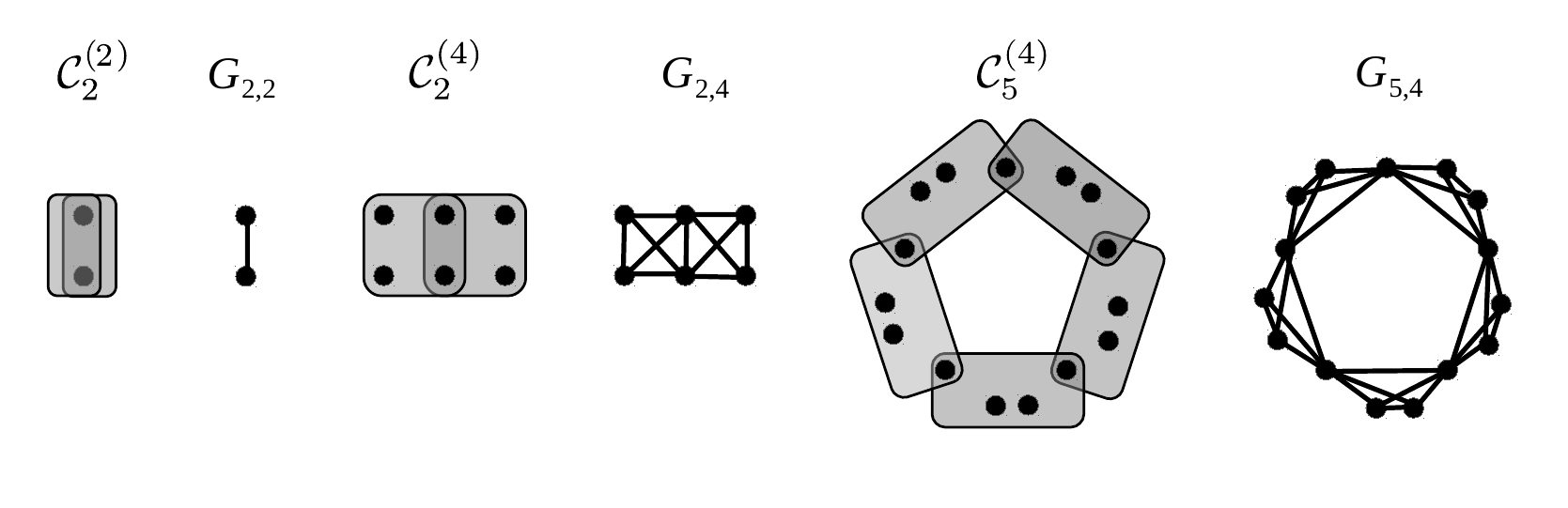}
\end{center}			
	\caption{Examples of hypergraphs $\C_{\ell}^{(r)}$ and their associated graphs $G_{\ell,r}$.}
	\label{fig:cycles}
\end{figure}

We can express the densities of cycles in the graphon $\vw$, defined in \eqref{eq:Volkswagen}, in terms of the graphon~$W$ as follows. 
  Assume that the vertices shared by consecutive $r$-cliques in $G_{\ell,r}$ have labels $1, \dots, \ell$. Denoting $F(x_1, \dots, x_r) = \prod_{1 \le i < j \le r} W(x_i, x_j)$, we write
  \begin{equation*}
    \vw(x_i,x_{i+1}) = \int F(x_i, x_{i+1}, y_{i,3}, \dots, y_{i, r}) \prod_{j =3}^r dy_{i,r},
  \end{equation*}
  whence, writing indices of $x_i$'s modulo $r$ and using Fubini's theorem,
  \begin{align}
  \notag t(C_\ell, \vw) &= \int \prod_{i = 1}^r \left( \int F(x_i, x_{i+1}, y_{i,3}, \dots, y_{i, r}) \prod_{j=3}^r dy_{i,j} \right) \prod_{i=1}^r dx_i \\
  \label{eq:slehacka} &= \int \prod_{i = 1}^r F(x_i, x_{i+1}, y_{i,3}, \dots, y_{i, r}) \;dx_1 \dots \;dx_r \;dy_{1,3}\dots \;dy_{1,r} \dots \;dy_{r,3} \dots \;dy_{r,r}.
  \end{align}
  For $\ell \ge 3$, a moment of thought reveals that the last integral is exactly $t(G_{\ell, r}, W)$, implying
\begin{equation}
	\label{eq:hypercycle_density}
	t\left (C_{\ell},\vw \right) = t\left(G_{\ell,r}, W\right), \qquad \mbox{for $\ell \ge 3$ and $r \ge 2$}\;.
\end{equation}
On the other hand, for $\ell = 2$ we have to be a bit more careful, noting that in the product \eqref{eq:slehacka} factor $W(x_1, x_2)$ appears twice, implying 
  \begin{equation}\label{eq:2_cycle_density}
    t\left( C_2, \vw \right) = t\left( K_r \oplus_2 K_r, W \right), \qquad \mbox{for $r \ge 2$} \;,
	\end{equation}
where, recall, $K_r \oplus_2 K_r$ is a multigraph obtained by gluing two cliques along two vertices (with a double edge between these vertices).

\subsection{Moment generating functions}

As we noted above, the main result of~\cite{BhaDiaMuk} expresses a particular random variable as a sum of squares of independent normal random variables, which is very similar to the expression appearing in our Theorem~\ref{thm:main}\ref{en:chisquare}. The following lemma asserts that such distributions are well-defined and gives a formula of their moment generating functions. Note that in \cite{BhaDiaMuk} only the convergence of the series in $L^1$ is asserted, but it clearly converges in $L^2$ as well, since $S_n := \sum_{j=1}^n \lambda_j \left( Z_j^2 - 1 \right)$ is easily seen to be a Cauchy sequence in the (complete) space~$L^2$. In particular this implies that $\Var [S_n] \to \Var [S]$.

\begin{lemma}[see \cite{BhaDiaMuk}, Proposition 7.1\footnote{Note that there is an error in the arXiv version of Proposition 7.1 in \cite{BhaDiaMuk}.}]
	\label{lem:BhaDiaMuk} 
	Let $\{\lambda_j\}_j$ be a finite or countable sequence of real numbers such that $\sum_j\lambda_j^2 < + \infty$ and  
	let $\{Z_j\}_{j}$ be independent standard normal random variables. Define a (possibly infinite) sum  $S=\sum_j \lambda_j (Z_{j}^{2}-1)$. Then $S$ converges almost surely and in 
	  $L^2$ and
	\begin{equation}
	  \label{eq:varS}
	  \Var [S] = 2 \sum_j \lambda_j^2.
	\end{equation} 
	Furthermore, the moment generating function $M_S(t) := \E\left[\exp(tS)\right]$ is finite for $|t| < \frac{1}{8}\left( \sum_j \lambda_j^2 \right)^{-1/2}$ and equals
	\[ 
	M_S(t) = \prod_{j} \frac{\exp(-\lambda_j t)}{\sqrt{1-2\lambda_j t}} \;. \] 
\end{lemma}

\subsection{Stein's method and the Wasserstein distance}\label{ssec:toolStein}

Stein's method is one of the most powerful tools for obtaining limit theorems (in fact it gives much more, by providing bounds on approximation errors). It is particularly efficient for sums of dependent random variables, when each variable depends on a relatively small (but not necessarily bounded) number of other variables. Here, we follow a survey article by Ross~\cite[Section~3]{Ross}. 

  The dependency structure of a collection of random variables $(Y_i : i \in I)$ is encoded by a \emph{dependency graph} $\mathcal{G}$ (on a vertex set $I$) if, for all $i\in I$, $Y_i$ is independent of the random variables $\{Y_j\}_{j\notin N_i}$, where $N_i$ is the neighborhood of $i$ in $\mathcal{G}$ (including $i$ itself). In general, a dependency graph is not uniquely determined, but in many scenarios, there exists a dependency graph, which naturally arises by capturing the obvious dependencies, and which is also minimal among all dependency graphs.

We shall work with the Wasserstein distance between two random variables, say $X$ and $Y$, which we denote by $\dw(X,Y)$. We do not need an exact definition --- which the reader can find on page~214 of~\cite{Ross} --- since the only property of the Wasserstein distance that we will use is that for $Z \sim \mathcal{N}(0,1)$ and a sequence $X_n$ of random variables $\dw(X_n,Z)\rightarrow 0$ implies that $X_n \dto Z$ (see~\cite[Section~3]{Ross}). We will use the following off-the-shelf bound for the Wasserstein distance based on Stein's method.
\begin{theorem}[see Theorem~3.6 in~\cite{Ross}]
  \label{thm:Stein}
  Let $(Y_i : i \in I)$ be a finite collection of random variables such that for every $i \in I$ we have $\E [Y_i] = 0$ and $\E [Y_i^4] < \infty$. Writing $\sigma^2 := \Var \left[\sum_{i \in I} Y_i\right]$, define $Q = \sum_{i \in I} Y_i/\sigma$. Let $\mathcal{G}$ be a dependency graph of $(Y_i : i \in I)$. Writing $D = \max_{i \in I} |N_i|$, we have 
  \begin{equation}
    \label{eq:vodni_kamen}
    \dw(Q,Z) \le \frac{D^{2}}{\sigma^{3}}\cdot\sum_i\mathbb{E}\left[|Y_i|^{3}\right]+\frac{\sqrt{28}D^{3/2}}{\sqrt{\pi}\cdot \sigma^{2}}\cdot\sqrt{\sum_i \mathbb{E}\left[Y_i^{4}\right]}.
  \end{equation}
\end{theorem}

\section{Proof of Theorem~\ref{thm:main}\ref{en:gaussian}}\label{sec:odrazko}

We will apply Theorem~\ref{thm:Stein} to a collection $(Y_R : R \in \binom {[n]}{r})$ of random variables indexed by $r$-sets $R$ of vertices. Writing $I_R$ for the indicator that $R$ induces a clique in $\GG(n,W)$, we define $Y_R = I_R - \E [I_R] = I_R - t_r$. In this notation we have $X_{n,r} - {\binom n r}t_r = \sum_{R \in \binom{[n]}{r}} Y_R$.

We let $\mathcal{G}$ be the natural dependency graph of the collection $(Y_{R}: R \in \binom{[n]}{r})$ with edges corresponding to pairs $R_1 R_2$ such that $R_1 \cap R_2 \neq \emptyset$.

The proof consists of two steps. In the first one we bound the maximum degree of $\mathcal{G}$. This step is fairly straightforward. In the second and less elementary step we compute the asymptotics of the variance of $\sum_R Y_R$. As we will see, the leading term in the asymptotics comes from pairs of cliques that share exactly one vertex. 

Let us proceeding with the details of the first step. Notice that in $\mathcal{G}$ every neighbourhood $N_R$ has the same size $D$, namely,
\begin{equation}\label{eq:D}
	D = \sum_{\ell = 1}^{r}\binom{r}{\ell}\binom{n-r}{r-\ell} = O(n^{r-1}).
\end{equation}

We now turn to the second step. Recall that $K_{r}\ominus_{j}K_{r}$ is a simple graph consisting of two $r$-cliques sharing $j$ vertices. Set $d_j = t(K_{r}\ominus_{j}K_{r} ,W)$, for $j\in [r]$. 
By Jensen's inequality we have that
\begin{align}
\label{eq:d1tr}
t_r^2=\left(\int_0^1 t_x(K_{r}^{\bullet},W) \; dx\right)^2 < \int_0^1 t_x(K_{r}^{\bullet},W)^2 \; dx = d_1,
\end{align}
where the strict inequality follows from the assumption that $W$ is not $K_r$-regular.

For disjoint $R_1, R_2$, variables $Y_{R_1}, Y_{R_2}$ are independent, while $|R_1 \cap R_2| = \ell \ge 1$ implies $\E[Y_{R_1}Y_{R_2}] = d_\ell - t_r^2$, and \eqref{eq:d1tr} implies $\rhor^2 = (d_1 - t_r^2)/((r-1)!)^2 > 0$. Since $\E [Y_R] = 0$ for every $R$, we have $\sigma_n^2 := \Var \left[\sum_R Y_R \right] = \sum_{R_1,R_2} \E [Y_{R_1} Y_{R_2}]$. Since the number of ordered pairs $R_1, R_2$ such that $|R_1 \cap R_2| = \ell$ is 
\[
  \binom{n}{\ell}\binom{n-\ell}{r-\ell}\binom{n-r}{r-\ell} \sim \frac{n^{2r - l}}{\ell!(r - \ell)!^2},
\]
we obtain 
\begin{align}
 \nonumber
 \sigma_n^2  &=\sum_{\ell=1}^{r}\frac{n^{2r - \ell}(1 + o(1))}{\ell!(r - \ell)!^2}\left(d_{\ell}-t_{r}^{2}\right)\\
 \nonumber
 &= \frac{n^{2r - 1}}{(r-1)!^2}\left(d_{1}-t_{r}^{2}\right)+\sum_{\ell=2}^{r}O(n^{2r-\ell})\\
\label{eq:Si}
&\sim \rhor^2  n^{2r-1}\;.
\end{align}

Writing $Q_n = \sum_{R \in \binom{[n]}{r}} Y_R / \sigma_n$, we are ready to apply Theorem~\ref{thm:Stein}. Crudely bounding each of the sums of moments on the right-hand side of \eqref{eq:vodni_kamen} by $\binom n r \le n^r$, and using \eqref{eq:D} and \eqref{eq:Si}, we obtain that $\dw(Q_n,Z) = O(n^{-1/2}) \to 0$. 
By the remark we made just before Theorem~\ref{thm:Stein}, we conclude that $Q_n \dto Z$. In view of \eqref{eq:Si} and Slutsky's theorem,
 \begin{equation*}
	 \frac{\sum_{R \in \binom{[n]}{r}} Y_R}{n^{r-1/2}} = \frac{\sigma_n}{n^{r-1/2}}\cdot Q_n \dto \rhor Z,
 \end{equation*}
 which completes the proof, in view of $\sum_{R \in \binom{[n]}{r}} Y_R = X_{n,r} - {\binom n r}t_r$.

\section{Proof of Theorem~\ref{thm:main}\ref{en:chisquare}}\label{sec:proofofC}
Define a random variable as on the right-hand side of~\eqref{eq:chsquareformula},
\begin{equation}\label{eq:Y}
  Y = \sigr \cdot Z+\frac{1}{2(r-2)!}\sum_{\lambda\in\Spec^-(\vw)}\lambda\cdot (Z_\lambda^{2}-1).
\end{equation}
We employ the method of moments, in the way it is described in Section~6.1 of \cite{Janson_Luczak_Rucinski}. For this it is enough to show that the moments of the random variable $({X_{n,r}-{\binom n r}t_r})/{n^{r-1}}$ converge to the corresponding moments of the random variable $Y$, and to verify that the moment generating function $M_Y(t) = \E [e^{tY}]$ is finite in some neighbourhood of zero (so that the distribution of $Y$ is determined by its moments).

Recall that for a standard normal random variable $Z$ we have $M_Z(x)=\exp\left(x^2/2\right)$ and hence $M_{\sigr \cdot Z}(x) = \exp\left(\frac{\sigr^{2}x^{2}}{2}\right)$.
On the other hand, Lemma~\ref{lem:BhaDiaMuk} tells us that the moment generating function of the second summand in \eqref{eq:Y} is $\prod_{\lambda\in\Spec^-(\vw)}\exp\left(-\frac{\lambda x}{2(r-2)!}\right)/\sqrt{1-\frac{\lambda x}{(r-2)!}}$. Since the moment generating function of a sum of independent random variables equals the product of the moment generating functions of individual generating functions, it follows that
\begin{equation}\label{eq:psh}
M_Y(x) = \exp\left(\frac{\sigr^{2}x^{2}}{2}\right)\prod_{\lambda\in\Spec^-(\vw)}\frac{\exp\left(-\frac{\lambda x}{2(r-2)!}\right)}{\sqrt{1-\frac{\lambda x}{(r-2)!}}}\, .
\end{equation}

On the other hand, to compute the moments of $({X_{n,r}-{\binom n r}t_r})/{n^{r-1}}$, we write
\[
X_{n,r}-\binom{n}{r}t_r = \sum_{R \in \binom {n}{r}} \left(I_R - t_r\right), 
\]
where, as in Section~\ref{sec:odrazko}, $I_R$ is the indicator of the event that the set of vertices $R$ induces a clique in $\GG(n,W)$.
Given an $m$-tuple $(R_1,\ldots,R_m)$ of (not necessary distinct) elements of ${\binom{[n]}{r}}$, let 
\begin{equation}\label{eq:defineDelta}
\Delta(R_1,\ldots,R_m) := \E\left[ \prod_{i=1}^{m}(I_{R_i}-t_r) \right] \, .
\end{equation}
Then  
\begin{equation}\label{eq:delta_terms} 
	\E\left[\left(X_{n,r}-{\binom n r}t_r\right)^m \right]  = \sum_{(R_1,\ldots,R_m)\in {\binom {[n]}{r}}^m} \Delta(R_1,\ldots,R_m) \, .
\end{equation} 
It is hence natural to analyze $\Delta(R_1,\ldots,R_m)$ depending on the structure of the tuple $(R_1,\ldots,R_m)$. Our plan is as follows. First, we introduce a certain family $\FamilyX(n,r,m)\subset {\binom{[n]}{r}}^m$ and show that $\Delta(R_1,\ldots,R_m)=0$ for each $(R_1,\ldots,R_m)\in \FamilyX(n,r,m)$. Next, we define another family $\mathfrak{F}(n,r,m)\subset {\binom{[n]}{r}}^m\setminus \FamilyX(n,r,m)$. In~\eqref{eq_counting_tuples} we will show that the set ${\binom{[n]}{r}}^m\setminus (\FamilyX(n,r,m)\;\cup\; \mathfrak{F}(n,r,m))$ has size $O(n^{(r-1)m-1})$, and hence the corresponding tuples have a contribution which is negligible with respect to the renormalization of~\eqref{eq:delta_terms} by $n^{-(r-1)m}$ (as, for example, in the statement of Claim~\ref{cl:generatingfunction}). We then classify the tuples in $\mathfrak{F}(n,r,m)$ according to the pattern in which they overlap so that in each class every tuple the contribution $\Delta(R_1,\ldots,R_m)$ is the same. By obtaining an explicit expression for this contribution  (in Claim~\ref{cl:ahoj}) and counting the number of tuples in each class (in Claim~\ref{claim:2}) we will arrive at a rather complicated asymptotic formula, which we interpret as a coefficient of some reasonably simple power series (see Claim~\ref{cl:generatingfunction}). Finally, with some luck we discover that this power series is exactly the moment generating function $M_Y$ (see Claim~\ref{cl:psequal}). Let us give details now.
	
Let $\FamilyX(n,r,m)\subset {\binom{[n]}{r}}^m$ be those $m$-tuples $(R_1,\ldots,R_m)$ for which we have $|R_i\cap (\cup_{j\neq i}R_j)|\le 1$ for some $i\in [m]$. Suppose now that $(R_1,\ldots,R_m)\in \FamilyX(n,r,m)$. Without loss of generality, suppose that $|R_m\cap (\cup_{j=1}^{m-1}R_j)|\le 1$. Assume first that $|R_m\cap (\cup_{j=1}^{m-1}R_j)|=1$, say $\{v\}=R_m\cap (\cup_{j=1}^{m-1}R_j)$. If we condition on $U_v = x$, the indicator $I_{R_m}$ becomes independent of $\left\{ I_{R_i} : i \in [m-1] \right\}$, and hence
\begin{equation}\label{eq:Kalahari}
	\Delta(R_1,\ldots,R_m)=\int_0^1 \E \left[ I_{R_m} - t_r \:\big|\: U_v = x\right]\cdot \E\left[\prod_{i=1}^{m-1}(I_{R_i}-t_r)\:\big|\:U_v = x\right]\;dx \, .
\end{equation}
Since $W$ is $K_r$-regular, we have $\E \left[ I_{R_m} - t_r \:\big|\: U_v = x \right] = t_{x}(K_r^\bullet,W) - t_r=0$ for almost every $x$. 
Therefore $\Delta(R_1,\ldots,R_m)=0$, as claimed. An even simpler calculation yields the same conclusion when $|R_m\cap (\cup_{j=1}^{m-1}R_j)|=0$. 
Hence, we can rewrite~\eqref{eq:delta_terms} as
\begin{equation}\label{eq:delta_termsecono}
\E\left[\left(X_{n,r}-t_r{\binom n r}\right)^m \right]  = \sum_{(R_1,\ldots,R_m)\in {\binom{[n]}{r}}^m\setminus \FamilyX(n,r,m)} \Delta(R_1,\ldots,R_m) \;.
\end{equation}

Every $m$-tuple $(R_1,\ldots,R_m)$ can be identified with a spanning hypergraph henceforth denoted $\Hy(R_1,\ldots,R_m)$, with vertex set $\cup_i R_i$ and 
hyperedge multiset $\{R_1,\ldots,R_m\}$. 

The following claim provides a sharp upper bound on the number of vertices in  $\Hy(R_1,\ldots,R_m)$, for 
$(R_1,\ldots,R_m)\in {\binom{[n]}{r}}^m\setminus \FamilyX(n,r,m)$.

\begin{claim}\label{claim:1}
	Suppose that $(R_1,\ldots,R_m)\in {\binom{[n]}{r}}^m\setminus \FamilyX(n,r,m)$. The number $v$ of vertices in the hypergraph $\Hy=\Hy(R_1,\ldots,R_m)$ satisfies $v\le (r-1)m$. The equality is attained if and only if each hyperedge in $\Hy$ contains exactly~2 vertices of degree~2 and all other vertices have degree~1. 
\end{claim}
\begin{proof}[Proof of Claim~\ref{claim:1}]
Let $k$ be the number of vertices in $\Hy$ of degree~1. 
Since $(R_1,\ldots,R_m)\not\in \FamilyX(n,r,m)$ we have that 
\begin{equation}\label{eq:fis}
k\le (r-2)m\;.
\end{equation}
Since $\Hy$ is spanning and $v-k$ vertices have degree at least~2, it follows that 
\begin{equation}\label{eq:fk}
k+2(v-k) \le \sum_{v\in V(\Hy)}\deg(v)=rm\; .
\end{equation}
Therefore 
\[v \leByRef{eq:fk} \frac{rm-k}{2}+k = \frac{rm+k}{2} \leByRef{eq:fis} \frac{rm + (r-2)m}{2} = (r-1)m\, , \]
and the first statement follows. To prove the second statement, 
suppose that the number of vertices in  $\Hy=\Hy(R_1,\ldots,R_m)$ satisfies $v= (r-1)m$.
Then~\eqref{eq:fis} and~\eqref{eq:fk} are both equalities and so $\Hy$ consists of $(r-2)m$ vertices of degree~1, and $m$ vertices of degree~2.  Since, by assumption, every $R_i$ contains at least two vertices of degree~2 it 
readily follows that it contains exactly two such vertices. The other implication is immediate. 
\end{proof}
Let $\mathfrak{F}(n,r,m)$ be those $(R_1,\ldots,R_m)\in {\binom{[n]}{r}}^m\setminus \FamilyX(n,r,m)$ for which the corresponding hypergraph $\Hy(R_1,\ldots,R_m)$ has $(r-1)m$ vertices. Since for each $(R_1,\ldots,R_m)\in {\binom{[n]} r}^m\setminus (\FamilyX(n,r,m)\cup \mathfrak{F}(n,r,m))$ we have $|\cup_i R_i|\le (r-1)m-1$, we can record each element of ${\binom{[n]} r}^m\setminus (\FamilyX(n,r,m)\cup \mathfrak{F}(n,r,m))$ by an $((r-1)m-1)$-set of $[n]$, and then by specifying to which of the sets $R_1,\ldots,R_m$ each element of that $((r-1)m-1)$-set is an element of. Thus,
\begin{equation}\label{eq_counting_tuples}
	\left| {\binom{[n]}{r}}^m\setminus (\FamilyX(n,r,m)\cup \mathfrak{F}(n,r,m)) \right|\le {\binom{n}{(r-1)m-1}}\cdot \left(2^m\right)^{(r-1)m-1}=O(n^{(r-1)m-1})\;.
\end{equation}

Since $|\Delta(R_1, \dots, R_m)| \le 1$, from~\eqref{eq:delta_termsecono} and \eqref{eq_counting_tuples} we infer
\begin{equation}\label{eq:joanBAEZistheBEST}
\E\left[\left(X_{n,r}-t_r{\binom n r}\right)^m \right]  = \sum_{(R_1,\ldots,R_m)\in \mathfrak{F}(n,r,m)} \Delta(R_1,\ldots,R_m)+O(n^{(r-1)m-1}) \;.
\end{equation}

Now, fix $(R_1,\ldots,R_m)\in \mathfrak{F}(n,r,m)$ and consider the hypergraph $\Hy=\Hy(R_1,\ldots,R_m)$. Notice that when $r=2$ then some edges in $\Hy$ may be double edges, but all hyperedges are simple when $r\ge 3$. 
Now, replace every $r$-edge, say $R$, of $\Hy$ by a 2-edge that consists of the vertices of $R$ having degree~$2$ and notice that this results in a $2$-regular multigraph, that is, a union of vertex-disjoint cycles and double edges. 
In particular, this implies that the hypergraph $\Hy$ is a union of vertex-disjoint loose cycles.

We now need to deal with the right-hand side of~\eqref{eq:defineDelta} for tuples in $\mathfrak{F}(n,r,m)$, that is for tuples corresponding to unions of loose cycles. We first note that we can factor it over cycles, namely if there is a partition $[m] = V_1 \cup \dots \cup V_k$ such that each $\Hy(R_i : i \in V_j)$ is a loose cycle, then the random variables $\prod_{i \in V_j}(I_{R_i} - t_r), j = 1, \dots, k$ are independent and hence
\begin{equation}
  \label{eq:cycle_factor}
\Delta(R_1,\ldots,R_m) = \prod_j \Delta(R_i : i \in V_j).
\end{equation}
In order to calculate the individual factors in \eqref{eq:cycle_factor}, the following claim will be useful.
\begin{claim}\label{claim:1.5}
For each $i,r\ge 2$, for any proper subhypergraph $\C\subset \C_{i}^{(r)}$, we have 
$\E\left[\prod_{R\in \C} I_R\right] = t_r^{|\C|}$.
\end{claim}
\begin{proof}[Proof of Claim~\ref{claim:1.5}]
We proceed by induction on the number of hyperedges of $\C$. The case when $\C=\emptyset$ is trivial. So suppose that $\C\neq\emptyset$.

\begin{figure}
\begin{center}
		\includegraphics[scale=0.8]{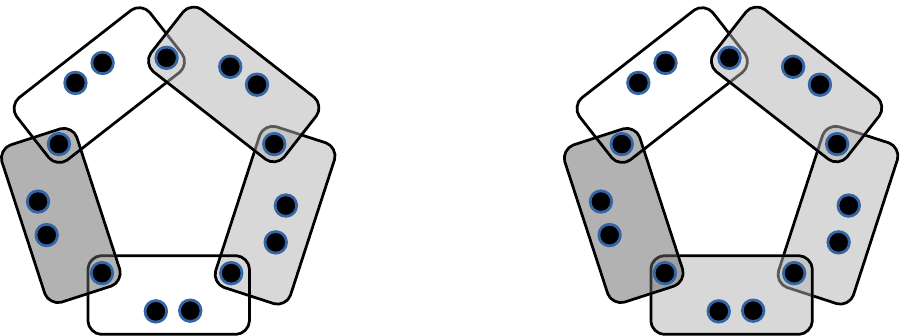}
\end{center}			
	\caption{Two examples of a hypergraph $\C^-$ (light grey) and an edge $S$ (dark grey) in a hypergraph~$\C_{5}^{(4)}$ in the proof of Claim~\ref{claim:1.5}.}
	\label{fig:cyclesS}
\end{figure}

Since $\C$ is a proper subhypergraph of $\C_{i}^{(r)}$, it contains a hyperedge $S\in \C$ such that for $\C^-:=\C\setminus \{S\}$ we have $|S\cap\bigcup \C^-|\le 1$ (here and below $\bigcup \Hy$ stands for the union of the hyperedges of $\Hy$). See Figure~\ref{fig:cyclesS}. Let us deal first with the case $|S\cap\bigcup \C^-|=1$, and let $v$ be the vertex shared by $S$ and $\bigcup \C^-$. By the same argument as in \eqref{eq:Kalahari}, we have
\begin{align*}
	\E\left[\textstyle\prod_{R\in \C} I_R\right]=\int_0^1 \E\left[I_S\:|\:U_v = x\right] \:\cdot\:\E\left[\textstyle\prod_{R\in \C^-} I_R\:|\:U_v = x\right] \;dx,
\end{align*}
By the $K_r$-regularity, we have $\E[I_S\:|\:U_v = \cdot] \equiv t_r$. Thus, using the induction hypothesis on $\C^-$, we conclude that 
\begin{align*}
	\E\left[\textstyle\prod_{R\in \C} I_R\right]=\int_0^1 t_r \E\left[\textstyle\prod_{R\in \C^-} I_R\:|\:U_v = x\right] \;dx = t_r \E\left[\textstyle\prod_{R\in \C^-} I_R\right]=t_r\cdot t_r^{|\C^-|}\;,
\end{align*}
as was needed. The case $|S\cap\bigcup \C^-|=0$ is even simpler:
\begin{align*}
\E\left[\textstyle\prod_{R\in \C} I_R\right]=\E\left[I_S\right] \:\cdot\:\E\left[\textstyle\prod_{R\in \C^-} I_R\right]=t_r\cdot t_r^{|\C^-|}\;.
\end{align*}
\end{proof}

Recall that for each $(R_1,\ldots,R_m)\in\mathfrak{F}(n,r,m)$, the hypergraph $\Hy(R_1,\ldots,R_m)$ is a union of loose cycles. Isomorphism classes of such hypergraphs can be encoded by a vector whose entry at position $i$ is the number of loose cycles of length $i$. More precisely, let us consider the following set of $(m-1)$-dimensional vectors,
\begin{equation}\label{eq:vectors}
 \mathcal{V}_m := \left\{\kk=(k_2,\ldots,k_m) \in \mathbb{N}_{0}^{m-1}: \sum_{i=2}^m ik_i =m \right\}.   
\end{equation}
Suppose that $\kk\in \mathcal{V}_m$. Let $\Hy_{\kk}^{(r)}$ denote the hypergraph formed by $k_i$ copies of $\C_i^{(r)}$ for each $i=2,\ldots,m$. 
\begin{claim}\label{cl:ahoj}
Suppose that $(R_1,\ldots,R_m)\in \mathfrak{F}(n,r,m)$ is an $m$-tuple for which $\Hy(R_1,\ldots,R_m)$ is isomorphic to $\Hy_{\kk}^{(r)}$, for some $\kk\in \mathcal{V}_m$. Then
\begin{equation}\label{eq:prod_formula} 
\Delta(R_1,\ldots,R_m) = \prod_{\ell=2}^m \left( t\left(G_{\ell,r},W\right)- t_{r}^{\ell} \right)^{k_\ell},
\end{equation}
where $G_{\ell, r}$ is the graph associated to $\C_{\ell}^{(r)}$, as defined in Section~\ref{ss_hypergraphs}.
\end{claim}
\begin{proof}[Proof of Claim~\ref{cl:ahoj}]
In view of \eqref{eq:cycle_factor}, it is enough to prove \eqref{eq:prod_formula} for an $\ell$-tuple $(Q_1,\ldots,Q_\ell)\in\mathfrak{F}(n,r,\ell)$ for which $\Hy(Q_1,\ldots,Q_\ell)$ is isomorphic to $\C_{\ell}^{(r)}$. We have
\begin{eqnarray*}
\Delta(Q_1,\ldots,Q_\ell) &=& \sum_{A\subseteq [\ell]} \left(-t_r \right)^{\ell-|A|} \cdot \E\left[ \prod_{j\in A} I_{Q_j} \right]\\
\justify{Claim~\ref{claim:1.5}}&=& \E\left[\prod_{j\in [\ell]} I_{Q_j} \right] + \sum_{A\subsetneq [\ell]}  \left(-t_r \right)^{\ell-|A|}t_{r}^{|A|} \\
&=&\E\left[\prod_{j\in [\ell]} I_{Q_j} \right] + \left( t_r -t_r\right)^\ell - t_r^\ell \\ 
&=& t\left(G_{\ell,r},W\right) - t_r^\ell\;.
\end{eqnarray*}
\end{proof}

\begin{claim}\label{claim:2}
Fix $\kk\in\mathcal{V}_m$. Then the number of $m$-tuples $(R_1,\ldots,R_m)\in{\binom{[n]} r}^m$ for which $\Hy(R_1,\ldots,R_m)$ is isomorphic to $\Hy_{\kk}^{(r)}$ is equal to 
\begin{equation}\label{eq:aut_formula} 
A(n,r,\kk):= \frac{ m! \cdot (n)_{(r-1)m} }{ \prod_{\ell=2}^m (2\ell ((r-2)!)^{\ell})^{k_{\ell}} \:\cdot\: k_{\ell}!  }.    
\end{equation}
\end{claim}
\begin{proof}[Proof of Claim~\ref{claim:2}] 
Suppose first that $r\ge 3$. 
Notice that the number of automorphisms of $\C_{\ell}^{(r)}$ equals $2\ell\cdot((r-2)!)^{\ell}$, and therefore the number of automorphisms of $\Hy_{\kk}^{(r)}$ satisfies
\[  \aut(\Hy_{\kk}^{(r)})=\prod_{\ell=2}^m ( 2\ell ((r-2)!)^{\ell})^{k_{\ell}} \:\cdot \: k_{\ell}! \;. \] 
As there are $\frac{(n)_{(r-1)m}}{\aut(\Hy_{\kk}^{(r)})}$ copies of $\Hy_{\kk}^{(r)}$ on $n$ vertices and each copy corresponds to $m!$ many $m$-tuples $(R_1,\ldots,R_m)$,  the proof of the case $r\ge 3$ is complete. 

The case $r=2$ is similar; the only difference being that the number of automorphisms of  $\C_{2}^{(2)}$ equals $2$ and that each copy of $\Hy_{\kk}^{(2)}$ corresponds to $\frac{m!}{2^{k_2}}$ many $m$-tuples. The details are left to the reader. 
\end{proof}

We now resume expressing $\E\left[\left(X_{n,r}-t_r{\binom n r}\right)^m \right]$, which we abandoned at~\eqref{eq:joanBAEZistheBEST}. Recall that $G_{\ell,r}$ is the graph associated with $\C_\ell^{(r)}$.
Adding~\eqref{eq:prod_formula} and \eqref{eq:aut_formula}, we get
\begin{eqnarray}
\nonumber
\E\left[\left(X_{n,r}-t_r\textstyle{\binom n r}\right)^m \right] &=& \sum_{\kk\in \mathcal{V}_m}  A(n,r,\kk)  \prod_{\ell=2}^m (t(G_{\ell,r},W)-t_{r}^{\ell})^{k_\ell} + O(n^{(r-1)m-1})\\
\label{eq:AnSe}
&=& n^{(r-1)m}m!\sum_{\kk\in\mathcal{V}_{m}}\prod_{\ell=2}^m\left( \frac{t(G_{\ell,r},W)-t_r^{\ell}}{2\ell((r-2)!)^{\ell}} \right)^{k_{\ell}}\frac{1}{k_{\ell}!}+O(n^{(r-1)m-1})\;.
\end{eqnarray}
For $\ell=2,3,\ldots$, let us write
\begin{equation}\label{eq:dees}
d_\ell := \frac{t(G_{\ell,r},W)-t_r^{\ell}}{2\ell((r-2)!)^{\ell}}\;.
\end{equation}

Treating $g(x) := \sum_{\ell = 2}^{\infty} d_\ell x^\ell$ as a formal power series, and substituting $y = g(x)$ in $e^y = \sum_{j = 0}^\infty y^j/j!$ we obtain another power series (since the free coefficient of $g$ is zero),
\begin{equation}\label{eq:psf}
  f(x) := \exp \left( g(x) \right)
= \exp\left(\sum_{\ell= 2}^\infty  d_\ell x^\ell\right)\;. 
\end{equation}

The following two claims are needed to show that $f(x)$ is the moment generating function of the limit of $(X_{n,r}-{\binom n r}t_r)/n^{r-1}$.

\begin{claim}\label{cl:generatingfunction}
For each $m\in\mathbb{N}$, as $n\rightarrow\infty$, the quantity $\frac{1}{m!}\cdot \E[\left(X_{n,r}-t_r{\binom n r}\right)^m/(n^{(r-1)m})]$ converges to $\llbracket x^m\rrbracket f(x)$, the coefficient of $x^m$ in the (formal) power series $f(x)$.
\end{claim}
\begin{proof}
We have
\begin{align*}
\llbracket x^m\rrbracket f(x)
&= \llbracket x^m\rrbracket \left( \exp\left(\sum_{\ell = 2}^\infty  d_\ell x^\ell \right)\right)
=\llbracket x^m\rrbracket \left( \sum_{j=1}^\infty \frac{1}{j!}\cdot \left(\sum_{\ell = 2}^\infty  d_\ell x^\ell \right)^j \right) \\
\justify{multinomial theorem}&=\sum_{j=1}^m \frac{1}{j!}\, \,  \sum_{\kk \in \mathcal{V}_m : k_2+\ldots+k_m=j}{\binom{j}{k_2,\cdots,k_m}}\cdot\prod_{\ell =2}^m d_\ell ^{k_\ell } \\
\justify{definition of the multinomial coefficient}&=\sum_{j=1}^m \, \,\sum_{\kk \in \mathcal{V}_m : k_2+\ldots+k_m=j}\, \, \prod_{\ell =2}^m \frac{d_\ell ^{k_\ell }}{k_\ell !} \\
&=\sum_{\kk\in\mathcal{V}_{m}}\, \prod_{\ell =2}^m d_\ell ^{k_{\ell }}\frac{1}{k_{\ell }!}\\
\justify{by~\eqref{eq:AnSe}}&=\lim_{n\rightarrow \infty} \frac{1}{m!}\cdot \frac{\E\left[\left(X_{n,r}-t_r{\binom n r}\right)^m\right]}{n^{(r-1)m}}\;.
\end{align*}
\end{proof}

Since $|d_i|\le ((r-2)!)^{-i}$, in particular, the sequence $|d_i|, i \ge 2$ is bounded. Therefore the series $\sum_i d_ix^i$ has positive radius of convergence, and $f(x)$ can be expanded as its Taylor series around zero.
	In the next claim, we show that the function $f$ equals the \mgf{} $M_Y$ defined in~\eqref{eq:psh}.

\begin{claim}\label{cl:psequal}
In some neighbourhood of zero we have $M_Y(x)=f(x)$.
\end{claim}
\begin{proof}[Proof of Claim~\ref{cl:psequal}]
Recall the definition \eqref{eq:dees} of $d_\ell$. For $\ell = 2$ we have
\begin{eqnarray*}
4((r-2)!)^2 \cdot d_2 &=&  t(K_r\ominus_2 K_r, W) - t_r^2 \\ 
&=&   t(K_r\ominus_2 K_r, W)  - t(C_2,\vw) + t(C_2,\vw) -t_r^2\\ 
\justify{\eqref{eq:2_cycle_density}} &=& t(K_{r}\ominus_{2}K_{r},W)-t(K_{r}\oplus_{2}K_{r},W) + t(C_2,\vw) - t_r^2 \\
\justify{\eqref{eq:cycle_eigenvalues}} &=& t(K_{r}\ominus_{2}K_{r},W) - t(K_{r}\oplus_{2}K_{r},W) + \sum_{\lambda\in\Spec^-(\vw)}\lambda^{2},
\end{eqnarray*}
which, in view of the definition~\eqref{eq:kissmysoul} of $\sigr^2$, implies 
\begin{equation}\label{eq:housle1}
  d_2 = \frac{\sigr^{2}}{2}+\frac{1}{4}\sum_{\lambda\in\Spec^-(\vw)}\left(\frac{\lambda}{ (r-2)! } \right)^{2}.
\end{equation}
For $\ell\ge 3$, from \eqref{eq:hypercycle_density} and \eqref{eq:cycle_eigenvalues} it follows
\begin{align}
\label{eq:housle2}
\begin{split}
d_\ell &= \frac{ t(G_{\ell,r}, W) - t_r^\ell }{2\ell((r-2)!)^\ell} \\ 
&= \frac{  t(C_\ell,\vw) -t_r^\ell}{ 2\ell((r-2)!)^\ell } 
= \frac{1}{2\ell}\sum_{\lambda\in\Spec^-(\vw)} \left(\frac{\lambda}{(r-2)!}\right)^\ell.
\end{split}
\end{align}

We are now ready to relate $f$ and $M_Y$.
In view of \eqref{eq:eigenvalues}, we have
\begin{equation}
\label{eq:eig_l2}
    \sum_{\lambda\in\Spec^-(\vw)} \lambda^2 < \infty
\end{equation}
and in particular
\begin{equation}
  \label{eq:eig_linf}
    \sup_{\lambda\in\Spec^-(\vw)} |\lambda| < \infty.
\end{equation}
Substituting~\eqref{eq:housle1} and~\eqref{eq:housle2} into~\eqref{eq:psf}, we obtain 
\begin{equation}\label{eq:fletna}
\log f(x) = \sum_{\ell=2}^\infty d_{\ell}x^{\ell} =  \frac{\sigr^{2}x^{2}}{2} + \sum_{\ell=2}^\infty\sum_{\lambda\in\Spec^-(\vw)}\frac{1}{2\ell}\left(\frac{\lambda x}{(r-2)!}\right)^{\ell}.
\end{equation}
In order to interchange the order of summation in~\eqref{eq:fletna}, we check a condition that allows applying Fubini's theorem, namely that \eqref{eq:fletna} remains finite, if we replace all summands by their absolute values. By \eqref{eq:eig_linf}, for $|x|$ small enough, the sequence $a_\lambda := \frac{\lambda x}{(r-2)!}$ satisfies $c:= \sup_\lambda |a_\lambda| < 1$. Therefore, for each $\ell\ge 2$,
\[
\sum_\lambda |a_\lambda|^\ell \le \sum_\lambda c^{\ell-2}a_\lambda^2 \ByRef{eq:eig_l2}{=} O(c^\ell),
\]
which implies $\sum_\ell \sum_\lambda |a_\lambda|^\ell/(2\ell) =\sum_\ell O(c^\ell) < \infty$, verifying the desired condition. So, changing the order of summation in~\eqref{eq:fletna}, we obtain 
\begin{eqnarray*}
\log f(x) = \sum_{\ell=2}^\infty d_{\ell}x^{\ell} &=& \frac{\sigr^{2}x^{2}}{2}+\sum_{\lambda\in\Spec^-(\vw)}\sum_{\ell=2}^\infty\frac{1}{2\ell}\left(\frac{\lambda x}{(r-2)!}\right)^{\ell}\\ 
\justify{Taylor's series}&=&\frac{\sigr^{2}x^{2}}{2}-\frac{1}{2}\sum_{\lambda\in\Spec^-(\vw)}\left(\ln\left(1-\frac{\lambda x}{(r-2)!}\right)+\frac{\lambda x}{(r-2)!}\right)\;.
\end{eqnarray*}
By exponentiating the above expression we easily obtain~\eqref{eq:psh}, thus completing the proof.
\end{proof}
We are now finished with the proof of Theorem~\ref{thm:main}\ref{en:chisquare}. Indeed, Claims~\ref{cl:generatingfunction} and~\ref{cl:psequal} imply that the $m$th moment of $(X_{n,r}-\binom n r t_r)/n^{r-1}$ converges to $m! \llbracket x^m \rrbracket f(x) =m! \llbracket x^m \rrbracket M_Y(x) = \E [Y^m]$ for every $m$. 
As argued at the beginning of Section~\ref{sec:proofofC}, this implies that $\frac{X_{n,r}-{\binom n r}t_r}{n^{r-1}}\;\dto\; Y$.

\section{Concluding remarks}\label{sec:concluding}
In this paper, we initiated the study of limit theorems for complete subgraph counts in $\GG(n,W)$. However, the results in this paper should be considered just first steps, and the area offers several obvious open problems.
\begin{itemize}
	\item Extend Theorem~\ref{thm:main} to sparser regimes. Recall that the central limit theorem for the count of $K_r$ in $\GG(n,p)$ holds for $p=p(n)$ as small as $p(n)\gg n^{-2/r-1}$, that is, as long as the expected number of $K_r$'s tends to infinity.
	
\item To model a sparse inhomogeneous random graph, fix a scaling factor $p =  p(n) \to 0$. Then $\GG(n, p\cdot W)$ is a sparse inhomogeneous random graph model. Note that then the assumption that $W$ is bounded from above by~1 can be relaxed somewhat. For example, the giant component of $\GG(n, W/n)$ is studied in~\cite{BoJaRi}. So, we suggest to obtain limit theorems for the count of $K_r$ (or other graphs) in $\GG(n, p\cdot W)$.
	
\item To strengthen the limit theorem obtained here to a local limit theorem. Even in the case of $\GG(n,p)$ this is a very difficult problem which was resolved only recently for cliques~\cite{Berk} and even more recently for general connected graphs~\cite{SahSaw}. Note that such a local limit theorem would have additional restrictions. For example, if $W$ is a graphon consisting of two constant-$1$ components of measure $\frac12$ each, then $X_{n,2}$ is of the form $\binom{k}{2} + \binom {n-k}{2}$, $k\in \mathbb N$, that is, not all integer values can be achieved (including those in the bulk of the distribution).
\end{itemize}

\end{document}